\documentclass[12pt, reqno]{amsart}
\usepackage{amsmath, amsthm, amscd, amsfonts, amssymb, graphicx, color}
\usepackage[bookmarksnumbered, colorlinks, plainpages]{hyperref}

\input{mathrsfs.sty}

\newtheorem{theorem}{Theorem}[section]
\newtheorem{lemma}[theorem]{Lemma}
\newtheorem{proposition}[theorem]{Proposition}
\newtheorem{corollary}[theorem]{Corollary}
\theoremstyle{definition}

\newtheorem{example}[theorem]{Example}

\theoremstyle{remark}
\newtheorem{remark}[theorem]{Remark}
\numberwithin{equation}{section}
\begin{document}

\title [Some extensions of the Young  and Heinz inequalities]{Some extensions of the Young and  Heinz inequalities for Matrices }

\author[ M. Hajmohamadi,   R. Lashkaripour, M. Bakherad   ]{ M. Hajmohamadi$^1$, R. Lashkaripour$^2$ and M. Bakherad$^3$}

\address{$^1$$^{,2}$$^{,3}$ Department of Mathematics, Faculty of Mathematics, University of Sistan and Baluchestan, Zahedan, I.R.Iran.}

\email{$^{1}$monire.hajmohamadi@yahoo.com}
\email{$^2$lashkari@hamoon.usb.ac.ir}
\email{$^{3}$mojtaba.bakherad@yahoo.com; bakherad@member.ams.org}


%~~~~~~~~~~~~~~~~~~~~~~~~~~~~~~~~~~~~~~~~~~~~~~~~~~~~~~~~~~~~~~~~~~~~~~~~~~~~~~~~~~~~~~~~~~~~~~~~~~~~~~~~~~~~~~~~~~~~~~~~~~~~~~~~~~~
\begin{abstract}
In this paper, we present some extensions of the Young and Heinz inequalities
for the Hilbert-Schmidt norm as well as any unitarily invariant norm. Furthermore,
we give some inequalities dealing with matrices. More precisely, for  two positive semidefinite matrices $A$ and $B$ we show that
 \begin{align*}
\Big\|A^{\nu}XB^{1-\nu}+A^{1-\nu}XB^{\nu}\Big\|_{2}^{2}&\leq\Big\|AX+XB\Big\|_{2}^{2}-
2r\Big\|AX-XB\Big\|_{2}^{2}\\&
\,\,\,\,\,-r_{0}\left(\Big\|A^{\frac{1}{2}}XB^{\frac{1}{2}}-AX\Big\|_{2}^{2}+
\Big\|A^{\frac{1}{2}}XB^{\frac{1}{2}}-XB\Big\|_{2}^{2}\right),
\end{align*}
where $X$ is an arbitrary $n\times n$ matrix, $0<\nu\leq\frac{1}{2}$, $r=\min\{\nu, 1-\nu\}$ and $r_{0}=\min\{2r, 1-2r\}$.

\end{abstract} \maketitle
%~~~~~~~~~~~~~~~~~~~~~~~~~~~~~~~~~~~~~~~~~~~~~~~~~~~~~~~~~~~~~~~~~~~~~~~~~~~~~~~~~~~~~~~~~~~~~~~~~~~~~~~~~~~~~~~~~~~~~~~~~~~~~~~~~~~
\section{Introduction}
Let $\mathcal{M}_n$ be the $C^*$-algebra of all $n\times n$ complex matrices and  $\langle\,\cdot\,,\,\cdot\,\rangle$ be the standard scalar
product in $\mathbb{C}^n$. A capital letter means an $n\times n$ matrix in  $\mathcal{M}_n$.  For Hermitian matrices $A, B$, we write  $A\geq 0$ if $A$ is positive semidefinite, $A>0$ if $A$ is positive definite, and $A\geq B$ if $A-B\geq0$.
\noindent A norm $|||\,.\,|||$ on $\mathcal{M}_n$ is called unitarily invariant norm if
$|||UAV|||=|||A|||$ for all $A\in\mathcal{M}_n$ and all unitary
matrices $U, V\in\mathcal{M}_n$. The Hilbert-Schmidt norm is defined
by $\|A\|_2=\left(\sum_{j=1}^ns_j^2(A)\right)^{1/2}$, where $s(A)=(s_1(A),\cdots, s_n(A))$
denotes the singular values of $A$, that is, the eigenvalues of the positive semidefinite matrix $|A|=(A^*A)^{1/2}$, arranged in the decreasing order with their multiplicities counted.
 This norm is unitarily invariant. It is known that if $A=[a_{ij}]\in\mathcal{M}_n$, then $\|A\|_2=\Big{(}\sum_{i,j=1}^n|a_{ij}|^2\Big{)}^{1/2}$. The trace norm of $A$  can be expressed as ${\rm tr}(|A|)=||A||_{1}=\sum_{j=1}^ns_j(A)$.\\
 The classical Young's inequality says that for positive real numbers $a, b$ and $0\leq\nu \leq1$, we have $a^{\nu}b^{1-\nu}\leq\nu a+(1-\nu)b$. When $\nu=\frac{1}{2}$, Young's inequality is the arithmetic--geometric mean inequality, $\sqrt{ab}\leq\frac{a+b}{2}$.

Zhao and Wu in \cite{zh}, refined the Young's inequality in the following form
 \begin{align}\label{2}
a^{1-\nu}b^{\nu}+S_1(\nu)
+r(\sqrt{a}-\sqrt{b})^{2}\leq(1-\nu)a+\nu b,
\end{align}
where
\begin{align}
S_1(\nu)=\left((-1)^{r_0}2\nu+(-1)^{r_0+1}\left[{r_0+1\over2}\right]\right)\Big(\sqrt[4]{b^{2-k}a^k}-&\sqrt[4]{a^{k+1}b^{1-k}}\Big)^2,
\end{align}
$0<\nu\leq1$, $r=\min \{\nu, 1-\nu\}$, $r_0=[4\nu]$ and $k=[2v]$. Here $[x]$ is the greatest integer less than or
equal to $x$. Also, they proved a reverse of \eqref{2} as follows
\begin{align}\label{4}
(1-\nu)a+\nu b\leq a^{1-\nu}b^{\nu}&+R(\sqrt{a}-\sqrt{b})^{2}-S_1(\nu),
\end{align}
where $0<\nu\leq1$ and $R=\max\{\nu, 1-\nu\}$. They  showed  if $a, b>0$ and $0<\nu<1$, then
\begin{align}\label{13}
(a^{1-\nu}b^{\nu})^{2}+r^{2}(a-b)^{2}+S_1(\nu)\leq ((1-\nu)a+\nu b)^{2},
\end{align}
and
\begin{align}\label{8}
((1-\nu)a+\nu b)^{2}\leq (a^{1-\nu}b^{\nu})^{2}+(1-\nu)^{2}(a-b)^{2}-S_1(\nu),
\end{align}
where $r=\min\{\nu, 1-\nu\}$.
Applying inequalities \eqref{13} and \eqref{8}
 we have the following inequalities:\\
If $0<\nu\leq\frac{1}{2}$, then
\begin{align}\label{10}
\nu^{2}(a^{2}+b^{2})-&(2\nu^{2}ab+2r_{0}a\sqrt{ab}-r_{0}(ab+a^{2}))\nonumber\\&
\leq ((1-\nu)a+\nu b)^{2}-(a^{1-\nu}b^{\nu})^{2}\nonumber\\&
\leq
(1-\nu)^{2}(a^{2}+b^{2})-(2(1-\nu)^{2}ab+r_{0}b\sqrt{ab}-r_{0}(ab+b^{2})),
\end{align}
If  $\frac{1}{2}<\nu<1$, then
\begin{align}\label{11}
(1-\nu)^{2}(a^{2}+b^{2})-&(2(1-\nu)^{2}ab+2r_{0}b\sqrt{ab}-r_{0}(ab+b^{2}))\nonumber\\&
\leq((1-\nu)a+\nu b)^{2}-(a^{1-\nu}b^{\nu})^{2}\nonumber\\&
\leq\nu^{2}(a^{2}+b^{2})-(2\nu^{2}ab+2r_{0}a\sqrt{ab}-r_{0}(ab+a^{2})),
\end{align}
where $r=\min\{\nu, 1-\nu\}$ and $r_{0}=\min\{2r, 1-2r\}$.\\
 The Heinz means are defined as $H_{\nu}(a, b)=\frac {a^{1-\nu}b^{\nu}+a^{\nu}b^{1-\nu}}{2}$
for $a, b>0$ and $0\leq \nu \leq 1$. These interesting means interpolate between the
geometric and arithmetic means. In fact, the Heinz inequalities assert that $\sqrt{ab}\leq H_{\nu}(a, b)\leq\frac{a+b}{2}$, where $a, b>0$ and $0\leq \nu \leq 1$.

A matrix version of Young's inequality \cite{And} says that if $A, B\in {\mathcal M_{n}}({\mathbb C})$ are positive semidefinite and $0\leq \nu\leq1$, then
\begin{align}\label{300}
s_{j}(A^{1-\nu}B^{\nu})\leq s_{j}((1-\nu) A+\nu B)
\end{align}
for $j=1, 2,..., n$. It follows from \eqref{300} that if $A, B\in {\mathcal M_{n}}$ are positive semidefinite and $0\leq \nu\leq1$, then a trace version of Young's inequality holds
\begin{align}\label{A}
{\rm tr}|A^{1-\nu}B^{\nu}|\leq {\rm tr}((1-\nu) A+\nu B).
\end{align}
A determinant version of Young's inequality says that \cite{HO}
\begin{align}\label{B}
{\rm det}\left(A^{1-\nu}B^{\nu}\right)\leq {\rm det}\left((1-\nu)A+\nu B\right).
\end{align}
In \cite{KO}, it is shown the Young inequality for arbitrary unitarily invariant norms as follows
\begin{align}\label{C}
|||A^{1-\nu}XB^{\nu}|||\leq (1-\nu)|||AX|||+\nu|||XB|||
\end{align}
 in which $A$, $B$ are positive semidefinite $n\times n$ and $0<\nu\leq 1$. Some mathematicians proved  several refinements of the Young and Heinz inequalities for matrices; see \cite{man, Mo, Ba} and references therein. Sababheh \cite{Sab} showed that for any $A, B, X\in {\mathcal M_{n}}$ such that $A$ and $B$ are positive semidefinite, the following relation holds
\begin{align}\label{m}
|||A^{1-\nu}XB^{\nu}|||&+\nu(|||AX|||+|||XB|||)\nonumber\\&\,\,-\Big(2\nu\sqrt{|||AX||||||XB|||}-r_{0}(\sqrt{|||AX|||}+
\sqrt[4]{|||AX||||||XB|||})^{2}\Big)\nonumber\\&
\leq (1-\nu)|||AX|||+\nu|||XB|||,
\end{align}
where $0<\nu\leq\frac{1}{2}$, $r=\min\{\nu, 1-\nu\}$ and $r_{0}=\min\{2r, 1-2r\}$.

Based on the refined and reversed Young inequalities \eqref{2} and \eqref{4}, Zhao and Wu \cite{zh}, proved that if $A, B, X\in {\mathcal{M}_n}$ such that $A$ and $B$ are two positive semidefinite matrices, then\\
$({\rm i})$ If $0<\nu\leq\frac{1}{2}$,
{\begin{align}\label{102}
r^{2}\Big\|AX-XB\Big\|_{2}^{2}+&r_{0}\Big\|A^{\frac{1}{2}}XB^{\frac{1}{2}}-AX\Big\|_{2}^{2}\nonumber\\&
\leq\Big\|(1-\nu)AX+\nu XB\Big\|_{2}^{2}-\Big\|A^{1-\nu}XB^{\nu}\Big\|_{2}^{2}\nonumber\\&
\leq R^{2}\Big\|AX-XB\Big\|_{2}^{2}-r_{0}\Big\|A^{\frac{1}{2}}XB^{\frac{1}{2}}-XB\Big\|_{2}^{2},
\end{align}}
$({\rm ii})$ if $\frac{1}{2}<\nu<1$,
{\begin{align}\label{103}
 R^{2}\Big\|AX-XB\Big\|_{2}^{2}+&r_{0}\Big\|A^{\frac{1}{2}}XB^{\frac{1}{2}}-XB\Big\|_{2}^{2}\nonumber\\&
\leq \Big\|(1-\nu)AX+\nu XB\Big\|_{2}^{2}-\Big\|A^{1-\nu}XB^{\nu}\Big\|_{2}^{2}\nonumber\\&
\leq r^{2}\Big\|AX-XB\Big\|_{2}^{2}-r_{0}\Big\|A^{\frac{1}{2}}XB^{\frac{1}{2}}-AX\Big\|_{2}^{2},
\end{align}}
where $r=\min\{\nu, 1-\nu\}$, $R=\max\{\nu, 1-\nu\}$ and $r_{0}=\min\{2r, 1-2r\}$.\\

In this paper, we generalized  some extensions of the Young and Heinz inequalities
for the Hilbert-Schmidt norm as well as any unitarily invariant norm. Also,
we give some inequalities dealing with matrices. Furthermore, we refine  inequalities \eqref{A}--\eqref{C}.
%####################################################################################################################################################################
\section{main results}
For our purpose  we need to following lemma.
\begin{lemma}\cite[Theorem 2 ]{Al}\label{phi}
Let $\phi$ be a strictly increasing convex function defined on an interval I. If $x, y, z$ and $w$ are points in I such that
$z-w\leq x-y$, where $w\leq z\leq x$ and $y\leq x$, then
\begin{align*}
(0\leq)\quad\quad   \phi(z)-\phi(w)\leq\phi(x)-\phi(y).
\end{align*}
\end{lemma}
\begin{theorem}
Let $\phi:[0, \infty)\rightarrow \mathbb{R}$ be a strictly  increasing convex function. If $a, b>0$, then\\
$({\rm i})$ For $0<\nu\leq\frac{1}{2}$,
\begin{align}\label{ab}
\phi\left(\nu(a+b)\right)-&\phi\left(2\nu\sqrt{ab}+2r_{0}\sqrt{a}\sqrt[4]{ab}-r_{0}(\sqrt{ab}+a)\right)\nonumber\\&
\leq \phi\left((1-\nu)a+\nu b\right)-\phi\left(a^{1-\nu}b^{\nu}\right)\nonumber\\&
\leq \phi\left((1-\nu)(a+b)\right)-\phi\left(2(1-\nu)\sqrt{ab}+2r_{0}\sqrt{b}\sqrt[4]{ab}-r_{0}(\sqrt{ab}+b)\right),
\end{align}
$({\rm ii})$ for $\frac{1}{2}<\nu<1$,
\begin{align}\label{abb}
\phi\left((1-\nu)(a+b)\right)&-\phi\left(2(1-\nu)\sqrt{ab}+2r_{0}\sqrt{b}\sqrt[4]{ab}-r_{0}(\sqrt{ab}+b)\right)\nonumber\\&
\leq \phi\left((1-\nu)a+\nu b\right)-\phi\left(a^{1-\nu}b^{\nu}\right)\nonumber\\&
\leq \phi\left(\nu(a+b)\right)-\phi\left(2\nu\sqrt{ab}+2r_{0}\sqrt{a}\sqrt[4]{ab}-r_{0}(\sqrt{ab}+a)\right),
\end{align}
where $r=\min\{\nu, 1-\nu\}$ and $r_{0}=\min\{2r, 1-2r\}$.
\end{theorem}
\begin{proof}
Let $0<\nu\leq\frac{1}{2}$. If we put $x=(1-\nu)a+\nu b$, $y=a^{1-\nu}b^{\nu}$, $z=\nu(a+b)$, $w=2r_{0}\sqrt{a}\sqrt[4]{ab}+2\nu\sqrt{ab}-r_{0}(\sqrt{ab}+a)$, $z^{\prime}=(1-\nu)(a+b)$ and $w^{\prime}=2(1-\nu)\sqrt{ab}+2r_{0}\sqrt[4]{ab}\sqrt{b}-r_{0}(\sqrt{ab}+b)$, then  $y\leq x$, $x\leq z'$. It follows from
\begin{align*}
2r_{0}\sqrt{a}\sqrt[4]{ab}+2\nu\sqrt{ab}&-r_{0}(\sqrt{ab}+a)\\&\leq
r_{0}(a+\sqrt{ab})+\nu(a+b)-r_{0}(\sqrt{ab}+a)\\&
  \qquad(\textrm {by the arithmetic-geometric mean})\\&
=\nu(a+b)\\&
\leq (1-\nu)a+\nu b
\end{align*}
and
\begin{align*}
2(1-\nu)\sqrt{ab}+2r_{0}\sqrt[4]{ab}\sqrt{b}&-r_{0}(\sqrt{ab}+b)\\&\leq (1-\nu)(a+b)+r_{0}(b+\sqrt{ab})-r_{0}(\sqrt{ab}+a)\\&
\qquad(\textrm {by the arithmetic-geometric mean})\\&
 =(1-\nu)(a+b),
 \end{align*}
 where    $w\leq z\leq x$, $w'\leq z'$. Using  inequalities \eqref{2} and \eqref{4} we have
\begin{align}\label{6}
\nu(a+b)-&\left(2\nu\sqrt{ab}+2r_{0}\sqrt{a}\sqrt[4]{ab}-r_{0}(\sqrt{ab}+a)\right)\nonumber\\&
\leq (1-\nu)a+\nu b-a^{1-\nu}b^{\nu}\nonumber\\&
\leq (1-\nu)(a+b)-\left(2(1-\nu)\sqrt{ab}+2r_{0}\sqrt{b}\sqrt[4]{ab}-r_{0}(\sqrt{ab}+b)\right).
\end{align}
 Hence
\begin{align*}
z-w\leq x-y\leq z^{\prime}-w^{\prime}.
\end{align*}
Applying Lemma \ref{phi} we reach inequality \eqref{ab}. Now, If $\frac{1}{2}<\nu<1$, then
\begin{align}\label{12}
(1-\nu)(a+b)-&\left(2(1-\nu)\sqrt{ab}+2r_{0}\sqrt{b}\sqrt[4]{ab}-r_{0}(\sqrt{ab}+b)\right)\nonumber\\&
\leq (1-\nu)a+\nu b-a^{1-\nu}b^{\nu}\nonumber\\&
\leq \nu(a+b)-\left(2\nu\sqrt{ab}+2r_{0}\sqrt{a}\sqrt[4]{ab}-r_{0}(\sqrt{ab}+a)\right).
\end{align}
In a similar fashion,  we have inequality \eqref{abb}.
\end{proof}
By taking $\phi(x)=x^{m}\,(m\geq 1)$, we have the next result.
\begin{corollary}\label{mo}
Let $a, b>0$ and $m\geq 1$. Then\\
$({\rm i})$ If $0<\nu\leq\frac{1}{2}$, then
\begin{align*}
\big(\nu(a+b)&\big)^{m}-\left(2\nu\sqrt{ab}+2r_{0}\sqrt{a}\sqrt[4]{ab}-r_{0}(\sqrt{ab}+a)\right)^{m}\\&
\leq \left((1-\nu)a+\nu b\right)^{m}-(a^{1-\nu}b^{\nu})^{m}\\&
\leq \left((1-\nu)(a+b)\right)^{m}-\left(2(1-\nu)\sqrt{ab}+2r_{0}\sqrt{b}\sqrt[4]{ab}-r_{0}(\sqrt{ab}+b)\right)^{m};
\end{align*}
$({\rm ii})$ if $\frac{1}{2}<\nu<1$, then
\begin{align*}
\left((1-\nu)(a+b)\right)&^{m}-\left(2(1-\nu)\sqrt{ab}+2r_{0}\sqrt{b}\sqrt[4]{ab}-r_{0}(\sqrt{ab}+b)\right)^{m}\\&
\leq \left((1-\nu)a+\nu b\right)^{m}-(a^{1-\nu}b^{\nu})^{m}\\&
\leq \left(\nu(a+b)\right)^{m}-\left(2\nu\sqrt{ab}+2r_{0}\sqrt{a}\sqrt[4]{ab}-r_{0}(\sqrt{ab}+a)\right)^{m},
\end{align*}
where $r=\min\{\nu, 1-\nu\}$ and $r_{0}=\min\{2r, 1-2r\}$.
\end{corollary}
In the following result, we show a refinement of the Heinz inequality.
\begin{corollary}\label{Hinz}
Let $\phi:[0, \infty)\rightarrow \mathbb{R}$ be a strictly  increasing convex function. If $a, b>0$, then
\begin{align*}
\phi(r(a+b))&-\phi(2r\sqrt{ab}+r_{0}\sqrt[4]{ab}(\sqrt{a}+\sqrt{b})-\frac{r_{0}}{2}(\sqrt{a}+\sqrt{b})^{2})\\&
\leq\phi(\frac{a+b}{2})-\phi(H_{\nu}(a, b))\\&\leq \phi(R(a+b))-\phi(2R\sqrt{ab}+r_{0}\sqrt[4]{ab}(\sqrt{a}+\sqrt{b})-\frac{r_{0}}{2}(\sqrt{a}+\sqrt{b})^{2})
\end{align*}
for $0\leq\nu\leq1$, $R=\max\{\nu, 1-\nu\}$, $r=\min\{\nu, 1-\nu\}$ and $r_{0}=\min\{2r, 1-2r\}$.
\end{corollary}
\begin{proof}
Let $0\leq\nu\leq1$.  By interchanging $a$ with $b$ in inequalities \eqref{6} and \eqref{12}, respectively, then we get\\
\begin{align}\label{16}
r(a+b)&-\Big(2r\sqrt{ab}+r_{0}\sqrt[4]{ab}(\sqrt{a}+\sqrt{b})-\frac{r_{0}}{2}(\sqrt{a}+\sqrt{b})^{2}\Big)\nonumber\\&
\leq\frac{a+b}{2}-H_{\nu}(a, b)\nonumber\\&
\leq R(a+b)-\Big(2R\sqrt{ab}+r_{0}\sqrt[4]{ab}(\sqrt{a}+\sqrt{b})-\frac{r_{0}}{2}(\sqrt{a}+\sqrt{b})^{2}\Big).
\end{align}
Now, we put $x=\frac{a+b}{2}$, $y=H_{\nu}(a, b)$, $z=r(a+b)$, $w=2r\sqrt{ab}+r_{0}\sqrt[4]{ab}(\sqrt{a}+\sqrt{b})-\frac{r_{0}}{2}(\sqrt{a}+\sqrt{b})^{2}$, $z^{\prime}=R(a+b)$ and $w^{\prime}=
2R\sqrt{ab}+r_{0}\sqrt[4]{ab}(\sqrt{a}+\sqrt{b})-\frac{r_{0}}{2}(\sqrt{a}+\sqrt{b})^{2}$. Using the arithmetic-geometric mean and  \eqref{16} we have $y\leq x$, $w\leq z\leq x$, $w^{\prime}\leq z^{\prime}$, $y\leq x\leq z^{\prime}$ and
\begin{align*}
z-w\leq x-y\leq z^{\prime}-w^{\prime}.
\end{align*}
Applying Lemma \ref{phi} we get the desired result.
\end{proof}
\begin{example}
 If we take $\phi(x)= x^{m}\,(m\geq1)$ in Corollary \ref{Hinz}, then for positive numbers $a$ and $b$ we reach the inequality
\begin{align*}
(r(a+b))^{m}&-(2r\sqrt{ab}+r_{0}\sqrt[4]{ab}(\sqrt{a}+\sqrt{b})-\frac{r_{0}}{2}(\sqrt{a}+\sqrt{b})^{2})^{m}\\&
\leq(\frac{a+b}{2})^{m}-(H_{\nu}(a, b))^{m}\\&
\leq (R(a+b))^{m}-(2R\sqrt{ab}+r_{0}\sqrt[4]{ab}(\sqrt{a}+\sqrt{b})-\frac{r_{0}}{2}(\sqrt{a}+\sqrt{b})^{2})^{m},
\end{align*}
where $0\leq\nu\leq1$, $R=\max\{\nu, 1-\nu\}$, $r=\min\{\nu, 1-\nu\}$ and $r_{0}=\min\{2r, 1-2r\}$.
\end{example}

%####################################################################################################################################################################

\section{ Some applications}
In this section, we apply numerical inequalities that we achieved in section $2$ for Hilbert space operators.
First, we improve the inequalities \eqref{A}, \eqref{B} and \eqref{C}.
To achieve this, we need the following lemmas.
\begin{lemma}\label{106}
Let $A, B\in {\mathcal M_{n}}$. Then
\begin{align*}
\sum_{j=1}^{n}s_{j}(AB)\leq \sum_{j=1}^{n}s_{j}(A)s_{j}(B).
\end{align*}
\end{lemma}
The next lemma is a Heinz-Kato type inequality for unitarily invariant norms that known in \cite{HO}.
\begin{lemma}\label{tr}
Let $A, B, X\in {\mathcal M_{n}}$ such that $A$ and $B$ are positive semidefinite. If $0\leq \nu\leq1$, then
\begin{align*}
|||A^{1-\nu}XB^{\nu}|||\leq |||AX|||^{1-\nu}|||XB|||^{\nu}.
\end{align*}
In particular,
\begin{align*}
 {\rm tr}|A^{1-\nu}B^{\nu}|\leq ({\rm tr}A)^{1-\nu}({\rm tr}B)^{\nu}.
 \end{align*}
 \end{lemma}
The third lemma is the Minkowski inequality for determinants that known in \cite{KIT}.
\begin{lemma}\label{det}
Let $A, B\in {\mathcal M_{n}}$ be positive definite. Then
\begin{align*}
{\rm det}(A+B)^{\frac{1}{n}}\geq {\rm det}A^{\frac{1}{n}}+{\rm det}B^{\frac{1}{n}}.
\end{align*}
\end{lemma}
In the next result we show an extension of inequality \eqref{m}.
\begin{theorem}
Let $A, B\in {\mathcal M_{n}}$ be positive definite. If $0<\nu\leq\frac{1}{2}$, then
\begin{align}\label{tr1}
\Big({\rm tr}|A^{1-\nu}B^{\nu}|\Big)^{m}&+\nu^{m}\Big({\rm tr}A+{\rm tr}B\Big)^{m}\nonumber\\&\,\,\,-\Big(2\nu(
{\rm tr}(A){\rm tr}(B))^{\frac{1}{2}}-r_{0}(({\rm tr}(A){\rm tr}(B))^{\frac{1}{4}}-({\rm tr}(A))^{\frac{1}{2}})^{2}\Big)^{m}\nonumber\\&
\leq\Big({\rm tr}((1-\nu)A+\nu B)\Big)^{m}
\end{align}
and if $\frac{1}{2}\leq \nu\leq 1$, then
\begin{align}\label{tr2}
\Big({\rm tr}|A^{1-\nu}B^{\nu}|\Big)^{m}&+(1-\nu)^{m}\Big({\rm tr}A+{\rm tr}B\Big)^{m}\nonumber\\&\,\,\,-\Big(2(1-\nu)(
{\rm tr}(A){\rm tr}(B))^{\frac{1}{2}}-r_{0}(({\rm tr}(A){\rm tr}(B))^{\frac{1}{4}}-({\rm tr}(B))^{\frac{1}{2}})^{2}\Big)^{m}\nonumber\\&
\leq\Big({\rm tr}((1-\nu)A+\nu B)\Big)^{m},
\end{align}
where $m=1, 2, \cdots$,  $r=\min\{\nu, 1-\nu\}$ and $r_{0}=\min\{2r, 1-2r\}$.
\end{theorem}
\begin{proof}
Let $0<\nu\leq\frac{1}{2}$. Then
\begin{align*}
\Big({\rm tr}|A^{1-\nu}B^{\nu}|\Big)^{m}&+\nu^{m}\Big({\rm tr}A+{\rm tr}B\Big)^{m}\\&\,\,\,-\Big(2\nu(
{\rm tr}(A){\rm tr}(B))^{\frac{1}{2}}-r_{0}(({\rm tr}(A){\rm tr}(B))^{\frac{1}{4}}-({\rm tr}(A))^{\frac{1}{2}})^{2}\Big)^{m}\\&
\leq \Big(({\rm tr}(A))^{1-\nu}({\rm tr}(B))^{\nu}\Big)^{m}+\nu^{m}\Big({\rm tr}A+{\rm tr}B\Big)^{m}\\&\,\,\,-
\Big(2\nu(
{\rm tr}(A){\rm tr}(B))^{\frac{1}{2}}-r_{0}(({\rm tr}(A){\rm tr}(B))^{\frac{1}{4}}-({\rm tr}(A))^{\frac{1}{2}})^{2}\Big)^{m}
 \\&\qquad\qquad\qquad \qquad \qquad \qquad\qquad \qquad \qquad   (\textrm {by Lemma}\, \ref{tr})\\&
\leq \Big((1-\nu){\rm tr}(A)+\nu {\rm tr}(B)\Big)^{m}      \qquad  \qquad(\textrm {by Corollary}\, \ref{mo}) \\&
=\Big({\rm tr}((1-\nu)A+\nu B)\Big)^{m}.
\end{align*}
Thus, we get  inequality \eqref{tr1}.  Using Corollary \ref{mo}, Lemma \ref{tr} and  with a same argument in the proof of \eqref{tr1}, we have \eqref{tr2} for $\frac{1}{2}\leq \nu\leq1$.
\end{proof}
\begin{theorem}
Let $A, B\in {\mathcal M_{n}}$ be positive definite and $0<\nu\leq\frac{1}{2}$. Then
\begin{align*}
{\rm det (A^{1-\nu}B^{\nu})^{m}}&+\nu^{mn}\Big({\rm det}A+{\rm det}B\Big)^{m}\\&\,\,-\Big(2\nu(
{\rm det}(A){\rm det}(B))^{\frac{1}{2}}-r_{0}(({\rm det}(A){\rm det}(B))^{\frac{1}{4}}-({\rm det}(A))^{\frac{1}{2}})^{2}\Big)^{m}\\&
\leq{\rm det}((1-\nu)A+\nu B)^{m}
\end{align*}
holds for $m=1,2,\cdots$ and $r_{0}=\min\{2\nu, 1-2\nu\}$.
\begin{proof}
\begin{align*}
{\rm det}((1-\nu)A+\nu B)^{m}&=\Big({\rm det}((1-\nu)A+\nu B)^{\frac{1}{n}}\Big)^{mn}\\&
\geq \Big({\rm det}((1-\nu)A)^{\frac{1}{n}}+{\det}(\nu B)^{\frac{1}{n}}\Big)^{mn}    \qquad (\textrm {by Lemma}\, \ref{det})\\&
=\Big((1-\nu){\det A^{\frac{1}{n}}}+\nu{\det B^{\frac{1}{n}}}\Big)^{mn}\\&
\geq \Big(({\det A^{\frac{1}{n}}})^{1-\nu}({\det B^{\frac{1}{n}}})^{\nu}\Big)^{mn}
+\nu^{mn}\Big({\rm det}A+{\rm det}B\Big)^{m}\\&\,\,\,-\Big(2\nu(
{\rm det}(A){\rm det}(B))^{\frac{1}{2}}-r_{0}(({\rm det}(A){\rm det}(B))^{\frac{1}{4}}-({\rm det}(A))^{\frac{1}{2}})^{2}\Big)^{m}\\&
 \qquad  \qquad\qquad\qquad  \qquad \qquad \qquad  \qquad(\textrm {by Corollary}\, \ref{mo}) \\&
 ={\rm det (A^{1-\nu}B^{\nu})^{m}}
 +\nu^{mn}\Big({\rm det}A+{\rm det}B\Big)^{m}\\&\,\,\,-\Big(2\nu(
{\rm det}(A){\rm det}(B))^{\frac{1}{2}}-r_{0}(({\rm det}(A){\rm det}(B))^{\frac{1}{4}}-({\rm det}(A))^{\frac{1}{2}})^{2}\Big)^{m}.
\end{align*}
\end{proof}
\end{theorem}
\begin{theorem}
Let $A, B\in {\mathcal M_{n}}$ be positive definite. Then
\begin{align*}
 \Big((1-\nu)|||AX|||&+\nu|||XB|||\Big)^{m}\geq
|||A^{1-\nu}XB^{\nu}|||^{m}+\nu^{m}\Big(|||AX|||+|||XB|||\Big)^{m}\\&
-\Big(2\nu(
|||AX||| |||XB|||)^{\frac{1}{2}}-r_{0}((|||AX||| |||XB|||)^{\frac{1}{4}}-(|||AX|||)^{\frac{1}{2}})^{2}\Big)^{m}
\end{align*}
where  $m=1,2, \cdots$, $0<\nu\leq\frac{1}{2}$ and $r_{0}=\min\{2\nu, 1-2\nu\}$.
\begin{proof}
Applying Lemma \ref{det} and  Corollary \ref{mo} we have
\begin{align*}
|||A^{1-\nu}XB^{\nu}||&|^{m}+\nu^{m}\Big(|||AX|||+|||XB|||\Big)^{m}\\&
\,\,\,-\Big(2\nu(
|||AX||| |||XB|||)^{\frac{1}{2}}-r_{0}((|||AX||| |||XB|||)^{\frac{1}{4}}-(|||AX|||)^{\frac{1}{2}})^{2}\Big)^{m}\\&
\leq \Big(|||AX|||^{1-\nu}|||XB|||^{\nu}\Big)^{m}+\nu^{m}\Big(|||AX|||+|||XB|||\Big)^{m}\\&
\,\,\,-\Big(2\nu(
|||AX||| |||XB|||)^{\frac{1}{2}}-r_{0}((|||AX||| |||XB|||)^{\frac{1}{4}}-(|||AX|||)^{\frac{1}{2}})^{2}\Big)^{m}\\&
 \qquad  \qquad\qquad\qquad\qquad\qquad\qquad\qquad \qquad(\textrm {by Lemma}\, \ref{det})\\&
\leq \Big((1-\nu)|||AX|||+\nu|||XB|||\Big)^{m}      \qquad\qquad(\textrm {by Corollary}\, \ref{mo}).
\end{align*}
\end{proof}
\end{theorem}
\begin{remark}
If $\frac{1}{2}\leq\nu\leq1$, then similarly,  we can prove the following inequalities
\begin{align*}
{\rm det}((1-\nu)&A+\nu B)^{m}\geq
{\rm det (A^{1-\nu}B^{\nu})^{m}}+(1-\nu)^{mn}\Big({\rm det}A+{\rm det}B\Big)^{m}\\&-\Big(2(1-\nu)(
{\rm det}(A){\rm det}(B))^{\frac{1}{2}}-r_{0}(({\rm det}(A){\rm det}(B))^{\frac{1}{4}}-({\rm det}(B))^{\frac{1}{2}})^{2}\Big)^{m},
\end{align*}
and
\begin{align*}
 \Big((1-\nu)|&||AX|||+\nu|||XB|||\Big)^{m}\geq
|||A^{1-\nu}XB^{\nu}|||^{m}+(1-\nu)^{m}\Big(|||AX|||+|||XB|||\Big)^{m}\\&
-\Big(2(1-\nu)(
|||AX||| |||XB|||)^{\frac{1}{2}}-r_{0}((|||AX||| |||XB|||)^{\frac{1}{4}}-(|||XB|||)^{\frac{1}{2}})^{2}\Big)^{m}
\end{align*}
for all positive definite matrices $A, B\in {\mathcal M_{n}}$, $m=1,2,\cdots$ and $r_{0}=\min\{2-2\nu, 2\nu-1\}$.
\end{remark}

%####################################################################################################################################################################
In \cite{kit}, the authors showed that
\begin{align*}
2|||A^{\frac{1}{2}}XB^{\frac{1}{2}}|||\leq{|||AX+XB|||},
\end{align*}
where $A, B$ are positive definite matrices and $X$ is an arbitrary matrix. Using this inequality, inequalities \eqref{102} and \eqref{103}, we have the next result.
\begin{proposition}
Let $A, B, X\in {\mathcal M_{n}}$ such that $A, B$ are positive semidefinite. Then\\
$({\rm i})$ If $0<\nu\leq\frac{1}{2}$, then
\begin{align}\label{104}
&r^{2}\Big\|AX+XB\Big\|_{2}^{2}-
4\left(r^{2}\Big\|A^{\frac{1}{2}}XB^{\frac{1}{2}}\Big\|_{2}^{2}+r_{0}\Big\|A^{\frac{3}{4}}XB^{\frac{1}{4}}\Big\|_{2}^{2}-
r_{0}\Big\|A^{\frac{1}{2}}XB^{\frac{1}{2}}+AX\Big\|_{2}^{2}\right)\nonumber\\&
\leq\Big\|(1-\nu)AX+\nu XB\Big\|_{2}^{2}-\Big\|A^{1-\nu}XB^{\nu}\Big\|_{2}^{2}\nonumber\\&
\leq R^{2}\Big\|AX+XB\Big\|_{2}^{2}-
\left(4(R^{2}\Big\|A^{\frac{1}{2}}XB^{\frac{1}{2}}\Big\|_{2}^{2}+\Big\|A^{\frac{1}{4}}XB^{\frac{3}{4}}\Big\|_{2}^{2})-
r_{0}\Big\|A^{\frac{1}{2}}XB^{\frac{1}{2}}+XB\Big\|_{2}^{2}\right);
\end{align}
$({\rm ii})$ if $\frac{1}{2}<\nu<1$, then
\begin{align}\label{105}
&R^{2}\Big\|AX+XB\Big\|_{2}^{2}-\left(4
\left(R^{2}\Big\|A^{\frac{1}{2}}XB^{\frac{1}{2}}\Big\|_{2}^{2}+r_{0}\Big\|A^{\frac{1}{4}}XB^{\frac{3}{4}}\Big\|_{2}^{2}\right)-
r_{0}\Big\|A^{\frac{1}{2}}XB^{\frac{1}{2}}-XB\Big\|_{2}^{2}\right)\nonumber\\&
\leq\Big\|(1-\nu)AX+\nu XB\Big\|_{2}^{2}-\Big\|A^{1-\nu}XB^{\nu}\Big\|_{2}^{2}\nonumber\\&
\leq r^{2}\Big\|AX+XB\Big\|_{2}^{2}-\left(4
\left(r^{2}\Big\|A^{\frac{1}{2}}XB^{\frac{1}{2}}\Big\|_{2}^{2}+r_{0}\Big\|A^{\frac{3}{4}}XB^{\frac{1}{4}}\Big\|_{2}^{2}\right)-
r_{0}\Big\|A^{\frac{1}{2}}XB^{\frac{1}{2}}+AX\Big\|_{2}^{2}\right),
\end{align}
where $r=\min\{\nu, 1-\nu\}$, $R=\max\{\nu, 1-\nu\}$ and $r_{0}=\min\{2r, 1-2r\}$.
\end{proposition}
\begin{proof}
Let $0<\nu\leq\frac{1}{2}$. Applying
\begin{align*}
\Big\|AX-XB\Big\|_{2}^{2}=\Big\|AX+XB\Big\|_{2}^{2}-4\Big\|A^{\frac{1}{2}}XB^{\frac{1}{2}}\Big\|_{2}^{2},
\end{align*}
\begin{align*}
\Big\|A^{\frac{1}{2}}XB^{\frac{1}{2}}-AX\Big\|_{2}^{2}=\Big\|A^{\frac{1}{2}}XB^{\frac{1}{2}}+AX\Big\|_{2}^{2}
-4\Big\|A^{\frac{3}{4}}XB^{\frac{1}{4}}\Big\|_{2}^{2}.
\end{align*}
%and using of this inequation
%{\footnotesize\begin{align*}
%R^{2}\Big\|AX-XB\Big\|_{2}^{2}-r_{0}\Big\|A^{\frac{1}{2}}XB^{\frac{1}{2}}-XB\Big\|_{2}^{2}\leq R^{2}\Big\|AX-XB\Big\|_{2}^{2}+r_{0}\Big\|A^{\frac{1}{2}}XB^{\frac{1}{2}}-XB\Big\|_{2}^{2}
%\end{align*}}
%in inequality \eqref{102}
 and inequality \eqref{102}, we get the first inequality. For $\frac{1}{2}<\nu\leq1$, we can prove the second form of inequalities in a similar fashion.
\end{proof}

%To similarly,  inequalities in \eqref{103} and by the inequation\\
%{\footnotesize\begin{align*}
%\nu^{2}\Big\|AX-XB\Big\|_{2}^{2}-r_{0}\Big\|A^{\frac{1}{2}}XB^{\frac{1}{2}}-AX\Big\|_{2}^{2}\leq \nu^{2}\Big\|AX-XB\Big\|_{2}^{2}+r_{0}\Big\|A^{\frac{1}{2}}XB^{\frac{1}{2}}-AX\Big\|_{2}^{2},
%\end{align*}}

Applying Lemma \ref{phi} and inequality \eqref{104}, we have the following theorem.
\begin{theorem}\label{fi}
Let $A,B,X\in \mathcal{M}_n$ such that $A$ and $B$ are positive semidefinite.
If $\phi: [0, \infty)\rightarrow {\mathbb R}$ is a strictly increasing convex function and $0<\nu\leq\frac{1}{2}$, then
{\footnotesize\begin{align*}
&\phi\left(r^{2}\Big\|AX+XB\Big\|_{2}^{2}\right)-
\phi\left(4\left(r^{2}\Big\|A^{\frac{1}{2}}XB^{\frac{1}{2}}\Big\|_{2}^{2}+r_{0}\Big\|A^{\frac{3}{4}}XB^{\frac{1}{4}}\Big\|_{2}^{2}\right)-
r_{0}\Big\|A^{\frac{1}{2}}XB^{\frac{1}{2}}+AX\Big\|_{2}^{2}\right)
\nonumber\\&
\leq\phi\left(\Big\|(1-\nu)AX+\nu XB\Big\|_{2}^{2}\right)-\phi\left(\Big\|A^{1-\nu}XB^{\nu}\Big\|_{2}^{2}\right)\nonumber\\&
\leq \phi\left(R^{2}\Big\|AX+XB\Big\|_{2}^{2}\right)-\phi
\left(4\left(R^{2}\Big\|A^{\frac{1}{2}}XB^{\frac{1}{2}}\Big\|_{2}^{2}+\Big\|A^{\frac{1}{4}}XB^{\frac{3}{4}}\Big\|_{2}^{2}\right)-
r_{0}\Big\|A^{\frac{1}{2}}XB^{\frac{1}{2}}+XB\Big\|_{2}^{2}\right),
\end{align*}}
where $r=\min\{\nu, 1-\nu\}$, $R=\max\{\nu, 1-\nu\}$ and $r_{0}=\min\{2r, 1-2r\}$.
\end{theorem}
\begin{remark}
Note that for $\frac{1}{2}<\nu<1$, we can get the similarly inequality.
\end{remark}
\begin{example}
If $\phi(x)=x^{\frac{m}{2}}\,\, (m\geq 2)$, then using Theorem \ref{fi} we have
\begin{align*}
&r^{m}\Big\|AX+XB\Big\|_{2}^{m}-
\Big(4r^{2}\Big\|A^{\frac{1}{2}}XB^{\frac{1}{2}}\Big\|_{2}^{2}+
r_{0}\Big\|A^{\frac{3}{4}}XB^{\frac{1}{4}}\Big\|_{2}^{2}-
r_{0}\Big\|A^{\frac{1}{2}}XB^{\frac{1}{2}}+AX\Big\|_{2}^{2}\Big)^{\frac{m}{2}}
\nonumber\\&
\leq\Big\|(1-\nu) AX+\nu XB\Big\|_{2}^{m}-\Big\|A^{1-\nu}XB^{\nu}\Big\|_{2}^{m}\nonumber\\&
\leq R^{m}\Big\|AX+XB\Big\|_{2}^{m}-
\Big(4\left(R^{2}\Big\|A^{\frac{1}{2}}XB^{\frac{1}{2}}\Big\|_{2}^{2}+\Big\|A^{\frac{1}{4}}XB^{\frac{3}{4}}\Big\|_{2}^{2}\right)-
r_{0}\Big\|A^{\frac{1}{2}}XB^{\frac{1}{2}}+XB\Big\|_{2}^{2}\Big)^{\frac{m}{2}}.
\end{align*}
\end{example}
Replacing a and b by their squares in inequality \eqref{2}, for $0<\nu\leq\frac{1}{2}$, we have
\begin{align}\label{109}
(a^{1-\nu}b^{\nu})^{2}+r_{0}(\sqrt{ab}-a)^{2}+r(a-b)^{2}\leq (1-\nu)a^{2}+\nu b^{2}.
\end{align}
Now, Applying \eqref{109}, we have the following lemma.
\begin{lemma}\label{hnz}
If $a, b\geq 0$ and $0\leq\nu\leq1$, then
\begin{align*}
(a^{1-\nu}b^{\nu}+a^{\nu}b^{1-\nu})^{2}+2r(a-b)^{2}+r_{0}[(\sqrt{ab}-a)^{2}+(\sqrt{ab}-b)^{2}]\leq(a+b)^{2},
\end{align*}
where $r=\min\{\nu, 1-\nu\}$ and $r_{0}=\min\{2r, 1-2r\}$.
\end{lemma}
\begin{proof}
We have
\begin{align*}
(a+b)^{2}-&(a^{1-\nu}b^{\nu}+a^{\nu}b^{1-\nu})^{2}\\&=a^{2}+b^{2}-a^{2\nu}b^{2(1-\nu)}-a^{2(1-\nu)}b^{2\nu}\\&
=(1-\nu)a^{2}+\nu b^{2}-a^{2(1-\nu)}b^{2\nu}+\nu a^{2}+(1-\nu)b^{2}-a^{2\nu}b^{2(1-\nu)}\\&
\geq r(a-b)^{2}+r_{0}(\sqrt{ab}-a)^{2}+r(a-b)^{2}+r_{0}(\sqrt{ab}-b)^{2}\\&
=2r(a-b)^{2}+r_{0}[(\sqrt{ab}-a)^{2}+(\sqrt{ab}-b)^{2}].
\end{align*}
It follows the desired result.
\end{proof}
Now, applying Theorem \ref{hnz}, we improve the Heinz inequality for the Hilbert-Schmidt norm as follows:
\begin{theorem}
Let $A, B, X\in {\mathcal M_{n}}$ such that $A$ and $B$ are positive semidefinite. If $0<\nu\leq\frac{1}{2}$, then
\begin{align*}
\Big\|A^{\nu}XB^{1-\nu}+A^{1-\nu}XB^{\nu}\Big\|_{2}^{2}&\leq\Big\|AX+XB\Big\|_{2}^{2}-
2r\Big\|AX-XB\Big\|_{2}^{2}\\&
\,\,-r_{0}\Big(\Big\|A^{\frac{1}{2}}XB^{\frac{1}{2}}-AX\Big\|_{2}^{2}+
\Big\|A^{\frac{1}{2}}XB^{\frac{1}{2}}-XB\Big\|_{2}^{2}\Big),
\end{align*}
where $r=\min\{\nu, 1-\nu\}$ and $r_{0}=\min\{2r, 1-2r\}$.
\begin{proof}
Since $A,B\geq0$, it follows that there are unitary matrices $U, V\in {\mathcal M_{n}}$ such that $A=UDU^{*}$ and $B=VEV^{*}$, where $D={\rm diag}(\lambda_{1},\cdots,\lambda_{n})$, $E={\rm diag}(\mu_{1},\cdots,\mu_{n})$ and $\lambda_{i}, \mu_{i}\geq0\,\,(i=1, 2,\cdots,n)$. If $Y=U^{*}XV=[y_{ij}]$, then
\begin{align*}
(A^{\nu}XB^{1-\nu}+A^{1-\nu}XB^{\nu})%=U(D^{\nu}YE^{1-\nu}+D^{1-\nu}YE^{\nu})U^{*}
=U((\lambda_{i}^{1-\nu}\mu_{j}^{\nu}+\lambda_{i}^{\nu}\mu_{j}^{1-\nu})y_{ij})U^{*},
\end{align*}
\begin{equation*}
AX+XB=U[(\lambda_{i}+\mu_{j})\circ Y]V^{*},\qquad \qquad AX-XB=U[(\lambda_{i}-\mu_{j})\circ Y]V^{*},
\end{equation*}
\begin{equation*}
A^{\frac{1}{2}}XB^{\frac{1}{2}}-AX=U[((\lambda_{i}\mu_{j})^{\frac{1}{2}}-\lambda_{i})\circ Y]V^{*}, \qquad %\qquad
A^{\frac{1}{2}}XB^{\frac{1}{2}}-XB=U[((\lambda_{i}\mu_{j})^{\frac{1}{2}}-\mu_{j})\circ Y]V^{*},
\end{equation*}
whence
\begin{align*}
\Big\|A^{\nu}XB^{1-\nu}+A^{1-\nu}XB^{\nu}\Big\|_{2}^{2}&=\Big(\sum_{i, j=1}^{n}(\lambda_{i}^{1-\nu}\mu_{j}^{\nu}+\lambda_{i}^{\nu}\mu_{j}^{1-\nu})^{2}|y_{ij}|^{2}\Big)\\&
\leq\sum_{i, j=1}^{n}(\lambda_{i}+\mu_{j})^{2}|y_{ij}|^{2}-2r\sum_{i, j=1}^{n}(\lambda_{i}-\mu_{j})^{2}|y_{ij}|^{2}\\&
\,\,-r_{0}\sum\Big[(\lambda_{i}^{\frac{1}{2}}\mu_{j}^{\frac{1}{2}}-\lambda_{i})^{2}+
(\lambda_{i}^{\frac{1}{2}}\mu_{j}^{\frac{1}{2}}-\mu_{j})^{2}\Big]|y_{ij}|^{2}\\&
\qquad  \qquad \qquad \qquad \qquad \qquad (\textrm {by Theorem}\, \ref{hnz})\\&
=\Big\|AX+XB\Big\|_{2}^{2}-2r\Big\|AX-XB\Big\|_{2}^{2}\\&
\,\,-r_{0}\Big(\Big\|A^{\frac{1}{2}}XB^{\frac{1}{2}}-AX\Big\|_{2}^{2}+
\Big\|A^{\frac{1}{2}}XB^{\frac{1}{2}}-XB\Big\|_{2}^{2}\Big).
\end{align*}
\end{proof}
\end{theorem}
Applying the triangle inequality and Lemma \ref{tr}  we have the following result.
\begin{proposition}\label{mm}
Let $A, B, X\in {\mathcal M_{n}}$ such that $A$ and $B$ are positive semidefinite. If $0<\nu\leq\frac{1}{2}$, then
\begin{align*}
|||A^{1-\nu}X&B^{\nu}+A^{\nu}XB^{1-\nu}|||\leq
(1-2\nu)(|||AX|||+|||XB|||)-\Big(2(2\nu\sqrt{|||AX||||||XB|||})\\&-
r_{0}\Big((\sqrt{|||AX|||}+\sqrt[4]{|||AX||||||XB|||})^{2}+
(\sqrt{|||XB|||}+\sqrt[4]{|||AX||||||XB|||})^{2}\Big)\Big)
\end{align*}
\begin{proof}
We have
\begin{align*}
|||A^{1-\nu}X&B^{\nu}+A^{\nu}XB^{1-\nu}|||\\&\leq|||A^{1-\nu}XB^{\nu}|||+|||A^{\nu}XB^{1-\nu}|||\\&
\leq|||AX|||^{1-\nu}|||XB|||^{\nu}+|||AX|||^{\nu}|||XB|||^{1-\nu}\\&
\leq |||AX|||+|||XB|||-2\nu(|||AX|||+|||XB|||)\\&\,\,\,-
\Big(2(2\nu\sqrt{|||AX||||||XB|||})-r_{0}\Big((\sqrt{|||AX|||}+\sqrt[4]{|||AX||||||XB|||})^{2}\\&\,\,\,+
(\sqrt{|||XB|||}+\sqrt[4]{|||AX||||||XB|||})^{2}\Big)\Big).
\end{align*}
\end{proof}
\end{proposition}
\bigskip
%===================================================================================================================================
\bibliographystyle{amsplain}

\end{document}